\documentclass{cDSS2e}

\usepackage[english]{babel}
\usepackage[retainorgcmds]{IEEEtrantools}
\usepackage{verbatim}
\usepackage{csquotes}
\usepackage{graphicx}
\usepackage{overpic}
\usepackage{transparent}
\usepackage{url}

\usepackage{color}

\newcommand{\Rb}{\mathbb R}

\newcommand{\Zb}{\mathbb Z}

\newcommand{\lb}{\left(}
\newcommand{\rb}{\right)}

\newcommand{\diff}[2]{\frac{\partial #1}{\partial #2}}

\newcommand{\enbrace}[1]{\lb #1 \rb}
\newcommand{\inv}[1]{#1^{-1}}
\newcommand{\setdef}[2]{\left\{ #1\ \left\lvert\ #2 \right.\right\}}
\newcommand{\seqz}[2]{\left\{ {#1}_{#2} \right\}_{#2\in\Zb}}
\newcommand{\seq}[1]{ \left\{ #1 \right\} }

\newcommand{\half}[1]{\frac{#1}{2}}

\newcommand{\vmod}[1]{\left| #1 \right|}
\newcommand{\linhull}[1]{\left< #1 \right>}

\newcommand{\normban}[1]{\left\| #1 \right\|}

\newcommand{\Nint}{[-N,N-1]} 
\newcommand{\Nintw}{[-N+1,N]} 
\newcommand{\Nints}{[-N,N]} 
\newcommand{\Ninta}{[0,2N]} 

\newcommand{\enkav}[1]{``#1''}

\newlength{\negskiplength}
\newcommand{\expi}[1]{\exp_{#1}^{-1} }
\newcommand{\ekM}[2]{ \exp_{p_{#1} } \enbrace{#2} }

\newcommand{\ebaseptM}[1]{\ekM{\baseptind}{#1}}
\newcommand{\tvM}{\ebaseptM{v}}

\newcommand{\mfdM}{ M } 
\newcommand{\tgtM}[1]{ T_{p_{#1}} \mfdM } 

\newcommand{\tgt}[1]{ \Rb^n } 
\newcommand{\mfd}{ \Rb^n } 
\newcommand{\ekP}[1]{ p_{#1}}   
\newcommand{\ek}[2]{ \ekP{#1} + {#2} }

\newcommand{\dmettime}{1}
\newcommand{\vfield}{X}
\newcommand{\dmet}{\Psi}
\newcommand{\dmetsupp}{\Theta}
\newcommand{\dmetk}[1]{\dmet } 
\newcommand{\flow}{\Phi}
\newcommand{\dmetclassletter}{$\mathcal{T}$}
\newcommand{\dmetclass}[1]{\mathcal{T}_{#1}}
\newcommand{\dmetMap}{\tilde{\dmet}}

\newcommand{\tzero}{1}
\newcommand{\Rep}{\mathrm{Rep}}
\newcommand{\Repe}[1]{\Rep\enbrace{#1} }
\newcommand{\smtclass}{1}
\newcommand{\eqconst}{L_1}

\newcommand{\dadd}{^{(d)}}

\newcommand{\baseptind}{-N}
\newcommand{\basept}{p_{\baseptind}}
\newcommand{\baseptk}[1]{p_{#1}}

\newcommand{\ebasept}[1]{\ek{\baseptind}{#1}}

\newcommand{\dm}{ \dmet_{ \varkappa } }
\newcommand{\dms}{ \dmetsupp_{ \varkappa } }

\newcommand{\dmz}{ \dmet_{0} }
\newcommand{\dmsz}{ \dmetsupp_{ 0 } }
\newcommand{\dmo}{ \dmet_{1} }

\newcommand{\dmod}{ \dmet_{1}\dadd }
\newcommand{\dmsod}{ \dmetsupp_{1}\dadd }

\newcommand{\npk}[1]{ \tilde{p}_{ #1} }
\newcommand{\nAk}[1]{ \tilde{A}_{ #1} }
\newcommand{\nnpk}[1]{ \hat{p}_{  #1} }
\newcommand{\npkd}[1]{ \tilde{p}_{ #1}\dadd }
\newcommand{\nAkd}[1]{ \tilde{A}_{ #1}\dadd }
\newcommand{\nnpkd}[1]{ \hat{p}_{  #1}\dadd }

\newcommand{\om}[1]{\Omega_{\varkappa,#1} }
\newcommand{\omod}[1]{\Omega_{1,#1}\dadd }

\newcommand{\oned}{\frac{1}{d}}
\newcommand{\tdiff}{\sigma\dadd} 
\newcommand{\tdiffoned}{s\dadd }
\newcommand{\tdiffonedfin}{s^* }
\newcommand{\reparam}{\beta\dadd}
\newcommand{\reparamz}{\alpha}
\newcommand{\repz}{a}

\newcommand{\shdt}{y\dadd}  
\newcommand{\shdtz}{x}  
\newcommand{\shadptz}{p^*}
\newcommand{\shadptnodadd}{\hat{p}}
\newcommand{\shadpt}{\hat{p}\dadd}
\newcommand{\vd}{v\dadd}

\newcommand{\tv}{\ebasept{v}}
\newcommand{\tvd}{\ebasept{\vd}}

\newcommand{\vres}{w\dadd}
\newcommand{\vpreres}{W\dadd}
\newcommand{\vfin}{w^*}

\newcommand{\ldz}{\lim_{d\to 0}}

\newcommand{\osmX}{ g_1 }

\newcommand{\taufun}{ \tilde{g} }

\newcommand{\constdmet}{C_{\varkappa}(d)}
\newcommand{\constdmetz}{C_{0}(d)}
\newcommand{\constdmeto}{C_{1}(d)}

\newcommand{\maxderiv}{Q_1}
\newcommand{\maxvfield}{Q_2}
\newcommand{\maxshearflow}{Q_3}
\newcommand{\maxdiv}{Q_4}

\newcommand{\constpknpk}{K_5}

\newcommand{\kaocond}{if $\varkappa = 1$ then}
\newcommand{\kazcond}{if $\varkappa = 0$ then}

\newcommand{\CR}{\mathcal{CR}}

\DeclareMathOperator{\dist}{dist}

\DeclareMathOperator{\LISP}{LISP}
\DeclareMathOperator{\Id}{Id}

\numberwithin{equation}{section}

 \newtheorem{statement}{Statement}[section]

\title{Lipschitz inverse shadowing for nonsingular flows}

\author{
Dmitry Todorov $^{\rm a}$$^\ast$
 \thanks{
   $^\ast$ Dmitry Todorov. Email: todorovdi@gmail.com  \vspace{6pt}} 
  \\  \vspace{9pt}  $^{\rm a}$
  {\em{Chebyshev laboratory, Saint-Petersburg State University, 
  14th line of Vasiljevsky Island, 29B, Saint-Petersburg, 199178, Russia}};
  }

\begin{document}

\maketitle

\begin{abstract}
    We prove that Lipschitz inverse shadowing for 
    nonsingular flows is equivalent to structural stability.
\end{abstract}


\begin{classcode}
\end{classcode}

\section{Introduction}


The notion of inverse shadowing was introduced by Pilyugin and Corless in
 \cite{PIL__APPROX_AND_REAL_TRAJ_FOR_GEN_DYN_SYS} and by Kloeden and Ombach in 
 \cite{KLOEDEN__HYP_HOME_AND_BISHAD}.
 They defined this notion for diffeomorphisms. Inverse shadowing for flows 
 was first introduced in \cite{LEE__INV_SHAD_FOR_EXP_FLOWS}.

 It is known that both Lipschitz shadowing and Lipschitz inverse shadowing properties 
 for diffeomorphisms are equivalent to structural stability
\cite{PILSDS, PILMELISP, PIL__LIP_SH_AND_SS_FOR_FLOWS}.
In \cite{LEE__INV_SHAD_FOR_SS_FLOWS} the authors proved that
structural stability for flows implies inverse shadowing.
In fact, they proved that structural stability implies Lipschitz inverse shadowing
although they did not use this term in their paper.

We prove that Lipschitz inverse shadowing for flows without rest points implies
structural stability.

\section{Definitions} \label{sec:defs}

Let $\vfield$ be a $C^{\smtclass}$ vector field 
on a Riemannian manifold
$\mfdM$ with metric $\dist$
and let $\flow$ be a flow generated by $X.$
We will only consider cases when $\mfdM$ is closed and when $\mfdM=\mfd$.

Let $d$ be a (small) positive number.
\begin{definition}
   We say that a mapping  
     $\dmet:\Rb \times\mfdM \to \mfdM$ is a $d$-method for flow $\flow$ if for any $t\in\Rb$, 
    \begin{equation}
        \dist\enbrace{\dmet(t+s,x),\flow(s,\dmet(t,x)) } < d
        ,\quad s\in [-\dmettime,\dmettime] , 
        \label{neq:dmetdef}
    \end{equation}
    and  $\dmet(0,x) = x$ for any  $x\in \mfdM.$
\end{definition}

We introduce several classes of $d$-methods, following \cite{LEE__INV_SHAD_FOR_SS_FLOWS}. 
Let $\dmet$ be a $d$-method. Denote by $(\mfdM)^\Rb$ the set of all functions
from $\Rb$ to $\mfdM.$ Consider a mapping
$$\dmetMap: \mfdM \to (\mfdM)^\Rb,$$
defined as 
$$\enbrace{\dmetMap(x)}(t)=\dmet(t,x),\quad x\in \mfdM,\ t\in\Rb.$$
\begin{itemize}
    \item We say that the $d$-method $\dmet$ belongs to the class $\dmetclass{p},$ if 
       the mapping  $\dmetMap$ is continuous in the pointwise-convergence topology on $(\mfdM)^\Rb.$
    \item We say that the $d$-method $\dmet$ belongs to the class $\dmetclass{o},$ if  
        the mapping $\dmetMap$ is continuous in compact-open topology on $(\mfdM)^\Rb.$
    \item We say that the $d$-method $\dmet$ belongs to the class $\dmetclass{c},$ if  
        it is continuous as a mapping of the form $\Rb\times \mfdM\to \mfdM.$
    \item We say that the $d$-method $\dmet$ belongs to the class $\dmetclass{h},$ if  
        it is a flow of some $C^1$ vector field $Y$ that satisfies $$d_0(X,Y)\leq d.$$ 
        Here $d_0$ is the $C^0$ metric on the space of $C^1$ vector fields on $\mfdM.$
    \item We say that the $d$-method $\dmet$ belongs to the class $\dmetclass{s},$ if  
        it is smooth as a mapping of the form $\Rb\times \mfdM\to \mfdM.$
\end{itemize}

\begin{remark}
    Our definition of a $d$-method is a definition of a family of 
    pseudotrajectories (see the definition of a $d$-pseudotrajectory of 
    a flow in \cite{PILSDS}). If a method belongs to one of the classes defined above,
     then the dependence of a pseudotrajectory on a point has some continuity 
    (smoothness) properties. 
\end{remark}

Define $\Rep$ as a set of all increasing homeomorphisms $\alpha:\Rb\to\Rb$ 
with $\alpha(0)=0.$ We introduce also the following notation:
\begin{equation*}
    \Repe{\delta} = \setdef{\alpha\in\Rep}{\vmod{\frac{\alpha(t)-
    \alpha(s)}{t-s} - 1 }\leq \delta,\quad t\neq s}.
\end{equation*}

\begin{definition}
    Let 
    \dmetclassletter be one of the classes of methods defined above.
    We say that a flow $\flow$ has Lipschitz inverse shadowing property 
    (we write $\flow\in \LISP$ in this case) with respect to the class \dmetclassletter if 
    for any point  $p\in \mfdM$ there exist constants  $L,d_0,$ such that  
    for any $d$-method $\dmet$ of the class \dmetclassletter with $d\leq d_0$ there exist
    a point $\shadptnodadd\in \mfdM$ and a reparametrisation $\alpha\in\Repe{Ld}$ such that
    \begin{equation}
        \dist(\flow(t,p),\dmet(\alpha(t),\shadptnodadd)) < Ld,\quad t\in \Rb.
    \end{equation}
\end{definition}

\begin{remark} \label{rem:invRep}
    It is easy to see that if $\delta$ is small enough and  
    $\alpha\in\Repe{\delta}$ then $\alpha^{-1}\in\Repe{2\delta}.$
\end{remark}

\begin{remark}
    Obviously the inclusion  $\flow\in \LISP$ with respect to the class $\dmetclass{c}$ 
    implies that $\flow \in \LISP$ with respect to the classes $\dmetclass{h}$ and $\dmetclass{s}.$
\end{remark}

\begin{remark}
    It is proved in \cite{LEE__INV_SHAD_FOR_SS_FLOWS} that the inverse shadowing properties 
    with respect to the classes of methods $\dmetclass{c},\ \dmetclass{o}$ and
    $\dmetclass{p}$ are equivalent. We do not discuss the class $\dmetclass{h}.$  
    Further away we write just about Lipschitz inverse shadowing
    without mentioning a class of methods, always meaning the class $\dmetclass{s}$.
\end{remark}

%
%
%
%
%
%
%


\section{Main Results}

The main result of the work is the following theorem:


\begin{theorem} \label{thm:mainth}
Let $\mfdM$ be a closed manifold and let the flow $\flow$ have no rest points. 
Then $\flow$ is structurally stable iff $\flow\in \LISP.$
\end{theorem}

\section{The idea of the proof}

We will use the same idea that the authors of \cite{PIL__LIP_SH_AND_SS_FOR_FLOWS} used.

Fix a point $p\in \mfdM.$ Denote $f=\flow(1,\cdot),$ 
$p_k=f^k(p),\ k\in\Zb,$  and  $A_k = Df(p_k),\ k\in\Zb.$
Let $P_k:\tgtM{k}\to \tgtM{k}$ be the orthogonal projections with kernels 
 $\linhull{X(p_k)}$ 
and let $V_k$ be the orthogonal complement to $X(p_k).$ Denote $B_k = P_{k+1}A_k : V_k\to V_{k+1}.$
    Consider the following system of difference equations
    \begin{equation}\label{eqs:tangenteqsnf}
        v_{k+1} = B_k v_k + b_{k+1},\ k\in\Zb.
    \end{equation}

Let $\CR$ be the chain-recurrent set of a flow  $\flow$.
The following is proved in \cite{PIL__LIP_SH_AND_SS_FOR_FLOWS}:
\begin{theorem}
    Let $\mfdM$ be a closed manifold.
    Suppose that 
    there exists a constant $\eqconst$ such that 
    for any point  $p$ and any bounded sequence $\seqz{b}{k}$ with entries from 
    the corresponding  
    $V_k,$ equations \eqref{eqs:tangenteqsnf} have a solution $\seqz{v}{k}$ 
    with the norm bounded by $\eqconst\normban{b}.$
    Then
    \begin{itemize}
        \item the set $\CR$ is hyperbolic; 
        \item the strong transversality condition is fulfilled.
    \end{itemize}
\end{theorem}

It is known ( e.g. \cite{SELGRADE__HYP_AND_CHAIN_RECUR} ), 
that the hyperbolicity of the chain-recurrent set implies 
Axiom $A'.$ In turn, Axiom $A'$ and strong transversality condition imply
structural stability.
So the \enkav{only if} part of Theorem \ref{thm:mainth} will be proved if
we prove that the conditions of the previous theorem are satisfied.  
The \enkav{if} part is proved in \cite{LEE__INV_SHAD_FOR_SS_FLOWS}.

The reason we consider only vector fields without singularities is
that unlike in \cite{PIL__LIP_SH_AND_SS_FOR_FLOWS} we cannot prove
that singularities are isolated in $\CR.$ If we knew that they are
than we could easily show that they are hyperbolic and prove 
Theorem \ref{thm:mainth} holds.

\begin{definition}
    We say that a flow $\flow$
    satisfy condition (UB) if 
\begin{enumerate}                                
    \item 
        the norms of the derivatives of the flow $\flow(x,s)$ with respect to initial data
        for $s\in[-1,1]$ are bounded. I.e., there exists 
$$ \maxderiv = \max_{x\in M,\ s\in [-1,1]} \normban{\diff{\flow(s,x)}{x}} .$$
    \item
        The lengths of vectors of the vector field $\vfield$ are bounded. I.e., there exists
       $$\maxvfield = \max_{x\in\mfd} \vmod{\vfield(x)}.$$
    \item  
        The reminder in the Taylor expansion of the flow $\flow$ is uniformly bounded. 

        Let the manifold $M$ be covered by a countable number of 
        open balls $V_i$ of the same radius each admitting a coordinate chart.
        
        There exists a monotonous function   
        $\osmX:[0,\infty)\to[0,\infty)$ with the following property. 
        If for $x\in M,\ t\in \Rb$ we fix a coorinate chart of 
        the ball $V_i$  
        containing the point $\flow(t,x)$ and denote the 
        representation of the flow $\flow$ in this chart by the same letter $\flow$,
        then for any  $h_1\in[-1,1]$, $h_2\in T_x M$, $\vmod{h_2}<1$, 
        the following estimate is fulfilled
        $$ \vmod{\flow(t+h_1,x+h_2) - \flow(t,x) -h_1 X(\flow(t,x) ) - 
        \diff{\flow(t,x)}{x}h_2  } \leq \osmX (\vmod{h_1}+\vmod{h_2}),$$
        where $\osmX (\vmod{h_1}+\vmod{h_2})/(\vmod{h_1}+\vmod{h_2})$ 
        tends to $0$ uniformly in $x$ for $(h_1,h_2)\to 0.$ 
        
        Here we assume that all the points $\flow(t+h_1,x+h_2)$
        belong to $V_i$, otherwise we can reparametrize the flow globally.
\end{enumerate}
\end{definition}
A flow generated by $C^{\smtclass}$ vector field on a closed manifold 
satisfies this condition. 

At first we prove the solvability of equations that are different from 
equations \eqref{eqs:tangenteqsnf}:

\begin{statement}  \label{st:solveqs}
    Let the flow have Lipschitz inverse shadowing property and satisfy condition (UB).
    Then there exists a constant  $\eqconst$ such that
    for any point $p\in \mfdM$ and any inhomogeneity  $\seqz{z}{k}$  
    that satisfies $\vmod{z_k}\leq 1,\ k\in\Zb,$ for any natural $N$
    there exists a sequence of real numbers
     $\seq{s_k}_{k\in\Nints}$ such that the system of equations  
\begin{equation}\label{eqs:tangenteqsx}
    x_{k+1} = A_k x_k + X(p_{k+1}) s_k + z_{k+1},\quad k\in\Nint
\end{equation}
has a solution $\seq{x^{(N)}_k}_{k\in\Nints}$ such that 
    $\vmod{x^{(N)}_k}\leq \eqconst,\ k\in\Nints.$
\end{statement} 

The proof of this statement is the main difficulty and
is given later.
Now we show how the solvability of equations \eqref{eqs:tangenteqsx} 
implies the solvability of equations  \eqref{eqs:tangenteqsnf}.
The proof of the next corollary is similar to the proof of Lemma 2 from 
\cite{PIL__LIP_SH_AND_SS_FOR_FLOWS}.

\begin{corollary} \label{thm:solveqsnf}
    For any sequence $\seqz{b}{k}$ that satisfies $\vmod{b_k} \leq 1$ and $b_k \in V_k$ 
    for all integer  $k$ there exists a solution  
      $\seqz{v}{k}$ of system of equations \eqref{eqs:tangenteqsnf} such that
    $\vmod{v_k}\leq L_1.$
\end{corollary}

\begin{proof}
    Take $z_k=b_k$ in equations \eqref{eqs:tangenteqsx}. 
    Statement \ref{st:solveqs} guarantees that there exists a constant
     $L_1$ such that for any integer $N$ there exists a sequence
     $\seq{s_k}_{k\in\Nints}$ such that system of equations  \eqref{eqs:tangenteqsx}
    has a solution $\seq{x_k}_{k\in\Nints}$ with the norm bounded by $L_1.$

    Fix $k\in\Nints$
    Note that $A_k X(p_k) = X(p_{k+1}).$ The definition of the projections  $P_k$ 
    implies the inclusion $(\Id-P_k)v\in \linhull{X(p_k)}$ for any  $v\in \tgtM{k}.$ 
    Thus $A_k(Id-P_k) = 0.$
    Now we have the equality
    $$ P_{k+1}A_k = P_{k+1}A_k P_k.$$

    Multiply equalities \eqref{eqs:tangenteqsx} by $P_{k+1}:$
    \begin{gather*} 
    P_{k+1} x_{k+1} = P_{k+1} A_k x_k + P_{k+1} X(p_{k+1}) s_k + P_{k+1} z_{k+1} = \\
     = P_{k+1} A_k P_k x_k + b_{k+1},\quad k\in\Nint. 
    \end{gather*}
    Thus the sequence $v^{(N)}_k = P_k x_k$ is a solution of equations 
    \eqref{eqs:tangenteqsnf} for a finite number of indices. Now we can
    pass to the limit as $N\to\infty.$ Due to the boundedness, $v^{(N)}_k$ has a limit  
     $v_k$ whose norm is also bounded by  $L_1.$

\end{proof}


\section{Proof of Statement \ref{st:solveqs} for $M = \Rb^n$} \label{sec:proofR}

At first we prove Statement \ref{st:solveqs} when $M = \Rb^n$ 
and the vector field $X$ satisfies condition (UB).  
This will allow us to demonstrate the idea of the proof with less amount of
technical details.

Fix a natural $N.$ We are going to use a method from \cite{PILMELISP}.
We will \enkav{inscribe} equations into a method and use 
the Lipschitz inverse shadowing property.

Let us construct a method of the class $\dmetclass{s}$ that contains
the equations \enkav{inside}.

Let $L,\ d_0$ be constants from the definition of the $\LISP,$ 
and $d$ be a positive number which we decrease later.

Fix positive numbers $r$ and $\tau.$ Assume that $\tau$ satisfies  $100\tau < 1-\tau$ 
and $\osmX(s) < s,\ s\in[0,\tau].$
Let $\gamma:[-\tau,\tau]\times[0,r]\to[0,1]$ be a function with the following
properties (see Fig. \ref{fig:gammagraph})
\begin{itemize}
    \item $\gamma(-\tau, \cdot) = 0;$
    \item $\gamma(\cdot, r) = 0;$
    \item $\gamma(t,s) = 1,\quad t\in[-\half{\tau},{\tau}],\ s\in[0, \half{r}];$
    \item $\gamma(t,s) \in (0,1),\quad (t,s)\not\in[-\half{\tau},{\tau}]\times [0, \half{r}];$
    \item $\gamma$ is smooth as a function of two arguments. 
\end{itemize}

\begin{figure}
        \centering
        \def\svgwidth{0.5\textwidth}
        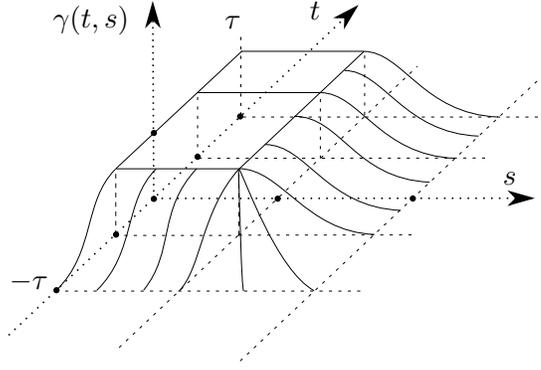
    \caption{The graph of  $\gamma$.}
    \label{fig:gammagraph}
\end{figure}


%
%

Let  $\taufun(d) $  be a function such that it tends to zero faster than $d$ and satisfies
$\osmX(\taufun(d))/d \to 0,\ d\to 0. $ 

\subsection{Construction of the method} \label{subsec:buildmetR}

Let $\varkappa$ be either  $0$  or  $1$ and
$k$ be an integer from  $\Nints.$  
Suppose that the functions  $\dm(k,\cdot),\ \dm(k+\tau,\cdot)$ are already defined. 
Denote
$$\npk{k} = \dm(k,\basept),\quad \nAk{k} = \diff{\flow(1-\tau,\dm(k+\tau,\basept) )}{x},$$
$$ \quad \nnpk{k+1} = \flow(1-\tau,\dm(k+\tau) ),\ \nnpk{\baseptind} = \basept.$$
If $k<N,$ define the function $\om{k+1}:
[-\tau,\tau]\times\tgt{\baseptind} \to \mfd $ in the following way:
\begin{gather*}
 \om{k+1}(s,v) =  \nnpk{k+1}  + 
\varkappa \nAk{k} \enbrace{ \dm(k,\tv) - \npk{k} + dz_k } + \\ 
+ \varkappa X ( \nnpk{k+1} ) s + d z_{k+1} - \varkappa d\nAk{k} z_k.
\end{gather*}
%
%
%
We also define the interpolation function
$\Gamma:\mfd\times\mfd\times \mfd \times [-\tau,\tau]\to \mfd :$
\begin{equation*}
    \Gamma(x,y,v,s) =  \gamma(s,\vmod{v})x + (1-\gamma(s,\vmod{v}) ) y  .  
\end{equation*}

Define the method $\dm$ in the following way 
 (see Fig. \ref{fig:dmetstructure}):

 \begin{gather*}
\dm(t,x) =  \flow(t,x),    \quad t\in \Rb,\quad x\notin B_r(\basept); \\
    \dm(t,x) =  \flow(t,x),  \quad t\in(-\infty,\tzero-\tau],\quad x\in \mfd; \\
    \dm(k+s,\tv) =   \Gamma \enbrace{ \om{k}(s,v) 
    ,\flow(1-\tau+s, \dm(k-1+\tau,\tv) ),v,s }, \\
    k\in[-N+1,N],\quad s\in[-\tau,\tau],\quad v\in B_r(0);\\
    \dm(k+\tau+s,\tv) =   \flow(s,\dm(k+\tau,\tv)), \\
    k\in[-N+1,N-1],\quad s\in[0,1-2\tau],\quad v\in B_r(0); \\
    \dm(N+\tau+s,\tv) =   \flow(s,\dm(N+\tau,\tv)),\quad
    s\in[0,\infty),\quad v\in B_r(0).
 \end{gather*}
Here  $B_r(x)$ is the ball of radius $r$ centered at $x$. 

\begin{figure}
        \centering                                 
        \def\svgwidth{\textwidth}
        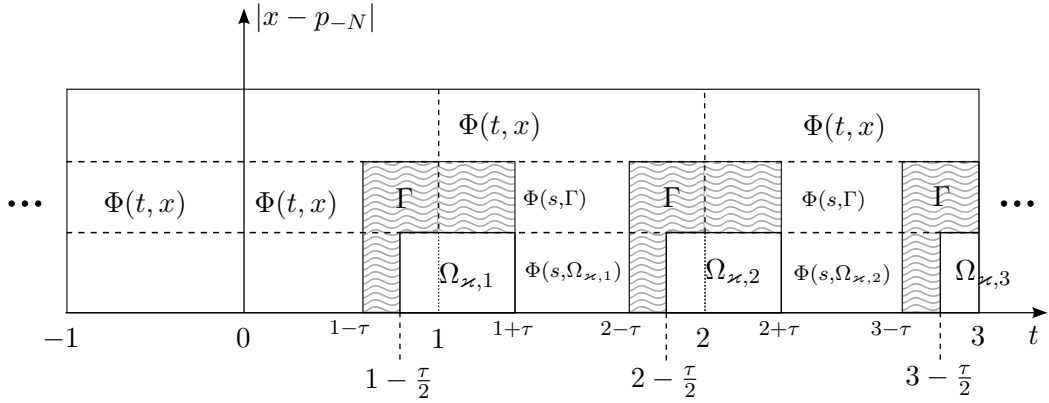
    \caption{Schematic picture of the method $\dm.$ 
    We added $N$ to the numbers below  $t$ axis for the ease of display. }
    \label{fig:dmetstructure}
\end{figure}

It is easy to see that the mapping we just defined is smooth.
Note that the mapping is a perturbation of the flow only at
points $x$ that are close to $\basept.$  
We emphasize that the mapping  $\dm$ depends both on   $d$ and $r.$
Note also that  $\npk{k}$ and $\nnpk{k}$ do not depend on  $\varkappa:$ 
$$ \npk{k} = \dmz(k,\basept) = \dmo(k,\basept)= \nnpk{k} + z_k,\quad k\in\Nints.$$
Thus we may rewrite the formula for $\om{k+1}$ in the following way:
\begin{IEEEeqnarray}{c}
    \om{k+1}(s,v) =  \nnpk{k+1}  + 
    \varkappa \nAk{k} \enbrace{ \dm(k,\tv) - \nnpk{k} } + \IEEEnonumber \\ 
    + \varkappa X ( \nnpk{k+1} ) s + d z_{k+1}, \quad k\in\Nint.
    \label{eq:omegadef}
\end{IEEEeqnarray}

We also define a supplementary mapping $\dms(t,x) = \dm(t+N,x).$
We need it because we have $t=0$ in the definition of a $d$-method 
and Lipschitz inverse shadowing.




\subsection{$d$-method conditions verification }


Let us prove that the mapping $\dms$ that we have constructed is a $\constdmet$-method
for some positive $\constdmet$ that is independent of  $N.$
It is obvious that $\dms(0,\cdot) = \Id.$ 
We need to check that inequalities \eqref{neq:dmetdef} are satisfied. 
The mapping  $\dms$ satisfies them iff $\dm$ does.
Therefore it is enough to estimate the following value 
for $t \in \Rb, \ s\in [-1,1]:$
\begin{equation*}                                    
    \delta_{t,s} = \vmod{\dm(t+s,\basept)-\flow(s,\dm(t,\basept) )} .
\end{equation*}
Since the value of the method $\dm(u,\basept)$ for
$u \in [-\infty,1-\tau]$ is equal to $\flow(u,\basept),$
we neet to estimate  $\delta_{t,s} $ only for
 $t,t+s \in[1-\tau,+\infty].$ 
 We will only consider the case of $s\geq 0,$ because for $s < 0$ 
 the estimates could be written in a similar way.

We call sections the sets $[k-\tau,k+\tau],\ k \geq 1$.
 Note that for  $t$ that is outside any section the following is true: 
 if we increase $s$ from $0$ to $1$ then $t+s$ will cross only one section  
 for a set of $s$ of nonzero measure. If  $t$ is in a section
 then the number of sections we will cross is equal to $2.$ 

Let $k$ be the biggest integer that is less than $t.$ I.e., $k=\lfloor t \rfloor$.
Consider the following cases 
\begin{itemize}
    \item $t$ and $t+s$ are outside sections, $s < 2\tau;$ 
    \item $t$ are outside sections, $t+s$ is inside a section;
    \item $t$ and $t+s$ are outside sections, $s > 2\tau;$ 
    \item $t$ is inside a section, $t+s$ is inside a section, $s < 2\tau;$
    \item $t$ is inside a section, $t+s$ is outside sections;
    \item $t$ is inside a section, $t+s$ is inside a section, $s > 2\tau.$
\end{itemize}

\subsubsection{$t$ and $t+s$ are outside sections, $s < 2\tau$ (Fig. \ref{fig:scheme1})}

\begin{figure}
        \centering                                 
        \def\svgwidth{0.5\textwidth}
        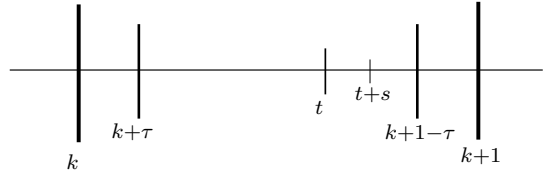
    \caption{$t$ and $t+s$ are outside sections, $s < 2\tau$  }
    \label{fig:scheme1}
\end{figure}

In this case 
$$ \dm(t+s,\basept) = \flow(s,\dm(t,\basept)) = \flow(t+s-k-\tau, \dm(k+\tau,\basept)),$$
thus $\delta_{t,s} = 0.$

\subsubsection{$t$ is outside sections, $t+s$ is inside a section (Fig. \ref{fig:scheme2})}
\label{sssec:toutsin} 

\begin{figure}
        \centering                                 
        \def\svgwidth{0.5\textwidth}
        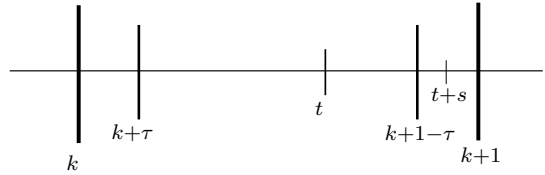
    \caption{$t$ is outside sections, $t+s$ is inside a section  }
    \label{fig:scheme2}
\end{figure}

Note that
\begin{equation}
    \vmod{\Gamma(x,y,v,s) - z} \leq \vmod{x-z} + \vmod{y-z},\quad x,y,z\in\mfd.
    \label{neq:Gamma} 
\end{equation}
 Denote $u = t+s - (k+\tau)$ and 
 $$\maxshearflow = \sup_{x\in\mfd,\ u\in[-1,1]\setminus\{0\}} 
 \frac{\vmod {\flow(u,x)- x} }{u}.$$
 Then
 \begin{gather*}
 \delta_{t,s} = \vmod{\dm(k+1 + t+s - (k+1)), \basept) - 
\flow(s ,\flow(t-(k+\tau),\dm(k+\tau,\basept)))} = \\  
 = \vmod{\Gamma (\om{k+1}(t+s-(k+1),0),\flow(u,\dm(1+\tau,\basept)))- 
\flow(u ,\dm(k+\tau,\basept))} \leq  \\  
 \leq \left| \flow(1-\tau,\dm(k+\tau,\basept)) + 
\varkappa X(\flow(1-\tau,\dm(k+\tau,\basept))) + \right. \\  
 \left.  -\varkappa\nAk{k} dz_k + dz_{k+1} - 
\flow(u,\dm(k+\tau,\basept)) \right| .
\end{gather*}

 It follows that

 \begin{itemize}
     \item \kaocond
        $$ \delta_{t,s}\leq\osmX(u) + \maxderiv d + d 
        \leq \osmX(\tau) + \maxderiv d+ d ;$$
    \item \kazcond
        $$ \delta_{t,s} \leq \maxshearflow u + \maxderiv d+ d 
        \leq \maxshearflow \tau +  d.$$
 \end{itemize}

\subsubsection{$t$ and $t+s$ are outside sections, $s > 2\tau$ (Fig. \ref{fig:scheme3})} 
\label{sssec:est}

\begin{figure}
        \centering                                 
        \def\svgwidth{0.6\textwidth}
        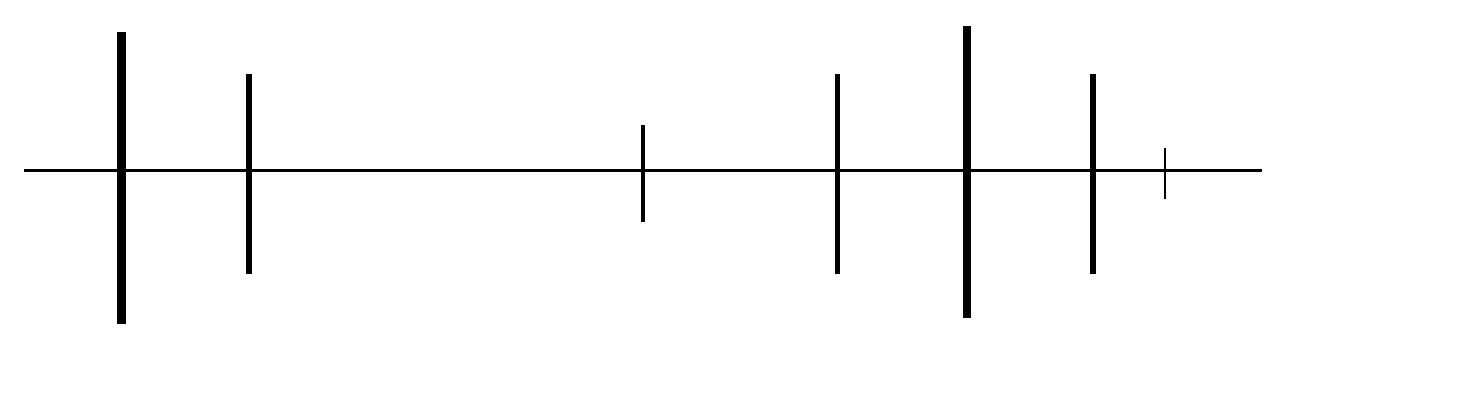
    \caption{$t$ and $t+s$ are outside sections, $s > 2\tau$ }
    \label{fig:scheme3}
\end{figure}

Denote $$\maxdiv = \sup_{x,y\in \mfd,\ x\neq y,\ s\in[-1,1]}
{ \frac{ \vmod{\flow(s,x)-\flow(s,y) } }{\vmod{x-y} } }$$
and $u = k+1+\tau.$
Then
\begin{gather*}
 \delta_{t,s} = \vmod{\flow(t+s-u,\dm(u,\basept) ) - 
\flow(t+s-u,\flow(u-t,\dm(t,\basept) ) ) } \leq  \\ 
 \leq \maxdiv\vmod{ \dm(u,\basept) - \flow(u-t,\dm(t,\basept) ) } .
 \end{gather*}
Now we may use the equality from the previous case. We get the following:
 
\begin{itemize}
     \item \kaocond
        $$ \delta_{t,s} \leq  \maxdiv ( \osmX(\tau) + \maxderiv d + d)   ;$$
    \item \kazcond
        $$ \delta_{t,s} \leq   \maxdiv (\maxshearflow \tau +  d )  .$$
 \end{itemize}

\subsubsection{$t$ is inside a section, $t+s$ is inside a section, $s < 2\tau$ (Fig. \ref{fig:scheme4})} 
\label{sssec:tinsin}

\begin{figure}
        \centering                                 
        \def\svgwidth{0.3\textwidth}
        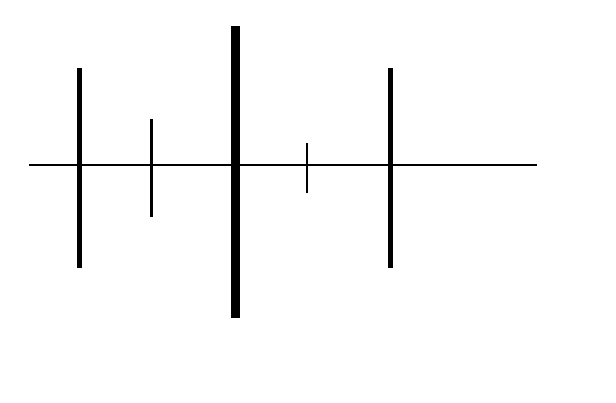
    \caption{$t$ is inside a section, $t+s$ is inside a section, $s < 2\tau$  }
    \label{fig:scheme4}
\end{figure}

Now let $k$ be an integer that is the closest to $t$. I.e., $k=[t]$.


Denote $u = k-\tau.$ Using the estimate from case \ref{sssec:toutsin} we may write

\begin{gather*}
 \delta_{t,s} \leq \vmod{\dm(t+s,\basept) - \flow(t+s-u,\dm(u,\basept) )} + \\
 + \vmod{ \flow(t+s-u,\dm(u,\basept) ) - \flow(s,\dm(t,p))}. 
 \end{gather*}
Thus

\begin{itemize}
     \item \kaocond
\begin{gather*}
         \delta_{t,s} \leq \osmX(t+s-u) + \maxderiv d + d + 
        s \maxdiv  \vmod{\flow(t-u,\dm(u,\basept) ) - \dm(t,\basept) } \leq \\
 \leq \osmX(\tau) + \maxderiv d +d + s \maxdiv (\osmX(t-u) + \maxderiv d + d) \leq \\ 
 \leq \osmX(\tau) + \maxderiv d +d + \maxdiv(\osmX(\tau) + \maxderiv d +d )   ; \\
\end{gather*}
    \item \kazcond
        \begin{gather*}
         \delta_{t,s} \leq \maxshearflow \tau + d  
        + s \maxdiv (\maxshearflow (t-u) + d ) \leq \\
         \leq \maxshearflow \tau + d  + \maxdiv (\maxshearflow \tau + d ). 
     \end{gather*}
 \end{itemize}

\subsubsection{$t$ is inside a section, $t+s$ outside sections (Fig. \ref{fig:scheme5})} 
\label{sssec:tintsout}

\begin{figure}
        \centering                                 
        \def\svgwidth{0.6\textwidth}
        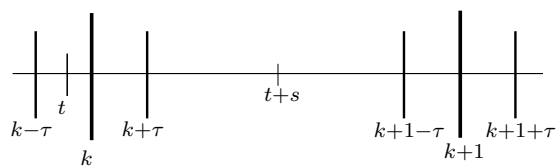
    \caption{ $t$ is inside a section, $t+s$ outside sections }
    \label{fig:scheme5}
\end{figure}

The next estimate follows from the previous one like in case \ref{sssec:est}: 

\begin{itemize}
     \item \kaocond
        $$ \delta_{t,s} \leq \maxdiv (\osmX(\tau) + \maxderiv d +d + \maxdiv(\osmX(\tau) + \maxderiv d +d ) ) ;$$
    \item \kazcond
        $$ \delta_{t,s} \leq \maxdiv(\maxshearflow \tau + d  + \maxdiv (\maxshearflow \tau + d ) ). $$
 \end{itemize}

\subsubsection{$t$ is inside a section, $t+s$ is inside a section, $s > 2\tau$ (Fig. \ref{fig:scheme6})} 
\label{sssec:tintsinagain}

\begin{figure}
        \centering                                 
        \def\svgwidth{0.6\textwidth}
        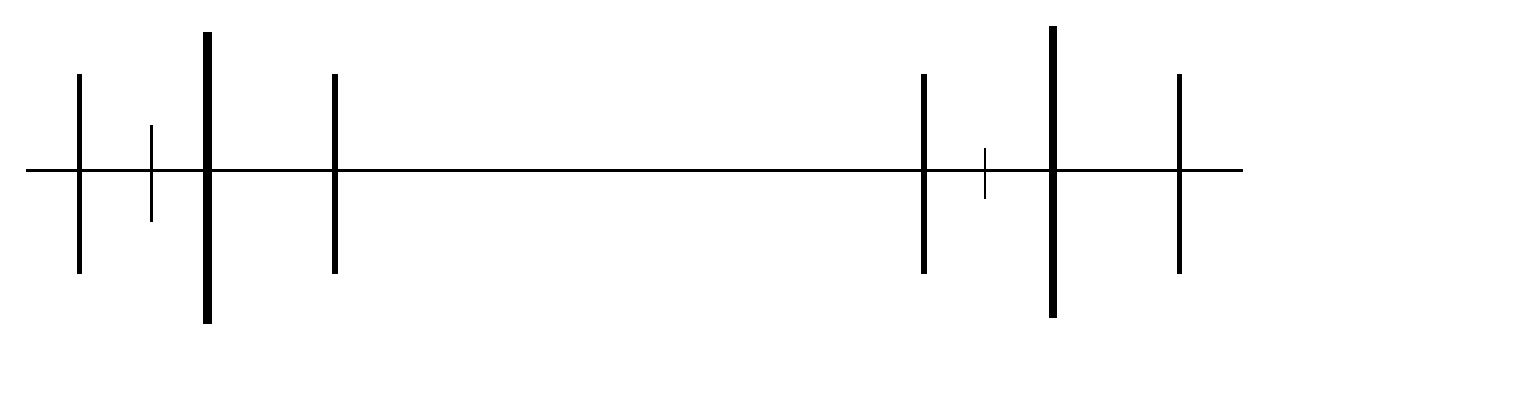
    \caption{$t$ is inside a section, $t+s$ is inside a section, $s > 2\tau$  }
    \label{fig:scheme6}
\end{figure}

Take
\begin{gather*}
 A = \vmod{\om{k+1}(k+1 - (t+s), 0) - \flow(t+s-(t+\tau),\dm(k+\tau,p)) }, \\
 B = \vmod{ \flow(t+s-(t+\tau),\dm(k+\tau,p)) - \flow(s,\dm(t,\basept)) } .
\end{gather*}
Then
$$ \delta_{t,s} \leq A + B.$$
We estimate the summands separately.
By definition,
\begin{gather*}
 \dm(t+s,p) = \Gamma\Bigl( \om{k+1}(k+1 - (t+s),0),  \\
 \flow(t+s-(k+\tau), \dm(k+\tau,\basept)) \Bigr).
\end{gather*}
Note that
$$ \flow(s,\dm(t,\basept)) = \flow(s-(k+\tau-t),\flow(t+\tau-t,\dm(t,p))).$$
Now we use condition (UB): 
\begin{gather*}
 \flow(k+\tau-t,\dm(t,p)) = \flow(k+\tau-t,\nnpk{k} + 
\varkappa X(\nnpk{k}) (t-k) + dz_k - \varkappa dA_kz_{k-1})= \\
 = \nnpk{k} + X(\nnpk{k})(k+\tau -t) + \varkappa X(\nnpk{k})(t-k) + 
dz_k - \varkappa dA_kz_{k-1} + \osmX(d+\tau) .
\end{gather*}
Moreover
$$ \dm(k+\tau,\basept) = \npk{k} + \varkappa X(\nnpk{k} )\tau + 
 dz_k - \varkappa dA_k z_{k-1}.$$
Thus we get
$$ \dm(k+\tau,\basept) - \flow(t+\tau-t,\dm(t,p)) = (1-\varkappa) X(\nnpk{k}) (t-k) 
+ \osmX(\tau + d) .$$
Consequently
\begin{gather*}
 B \leq \normban{ \diff{\flow(t+s-(t+\tau),\cdot)}{x} 
\enbrace{ \dm(k+\tau,\basept)) } } \cdot \\
\cdot \vmod{\dm(k+\tau,\basept) - \flow(t+\tau-t,\dm(t,p))} + \\  
 + \osmX(\vmod{\dm(k+\tau,\basept) - \flow(t+\tau-t,\dm(t,p))}) \leq \\
 \leq \maxderiv ( (1-\varkappa) \maxvfield \tau + \osmX(\tau+d) ) + \osmX(\tau+d).
\end{gather*}

We may estimate $A$ like we did in case \ref{sssec:toutsin}. Therefore  
\begin{itemize}
     \item \kaocond
$$ \delta{t,s} \leq \osmX(\tau) + \maxderiv d+ d +
 \maxderiv (  \osmX(\tau+d) ) + \osmX(\tau+d);
         $$
    \item \kazcond
$$
        \delta{t,s} \leq \maxshearflow \tau +  d
 \maxderiv ( \maxvfield \tau + \osmX(\tau+d) ) + \osmX(\tau+d).
        $$
 \end{itemize}


\subsubsection{The final estimate}

Take $\tau = \taufun(d).$ 
Our estimates imply that $\dm$ is a $\constdmet$-method with
\begin{itemize}
    \item \kaocond
$$\constdmeto =  K_1 d + K_2 \osmX(\taufun(d));$$
    \item \kazcond
$$\constdmetz =  K_3 d + K_4 \taufun(d).$$
\end{itemize}
Here $K_i,\ i=1..4$ are some constants that are independent of $N.$
This means that we can decrease $d$ so that 
 \begin{equation}\label{neq:shadtowork}
     L\constdmet < \min\enbrace{d_0, \half{r} }. 
 \end{equation}

  Moreover if $\varkappa = 1$ then we can guarantee 
  that the following inequality is satisfied
  $$
  L\constdmeto \leq  \frac{\tau}{4N}.
  $$

    \subsection{Shadowing of an exact trajectory by a trajectory of the method}
    
    At first we estimate the value $\delta_k = \vmod{\npk{k} - \baseptk{k}},\ k\in\Nints.$
    For $k=0$ we obviously have  
    $\vmod{ \basept - \npk{\baseptind} } \leq L\constdmetz.$
    The Lipschitz inverse shadowing property guarantees that there 
    exists a point $\shadptz$ and a reparametrization $\reparamz\in\Repe{L\constdmetz}$ 
    such that
    $$ \vmod{\flow(t,\basept)- \dmsz(\reparamz(t), \shadptz)} 
    \leq L\constdmetz \leq \half{r},\quad t\in\Rb. $$
        Fix $k\in \Nints.$  
        Denote 
        $$\repz_k = \reparamz(k+N)-N,\quad 
        \shdtz_k = \dmsz(\repz_k+N,\shadptz)=\dmz(a_k,\shadptz)  .$$
        Thus
            $$ \delta_k  \leq  \vmod{ \baseptk{k} - \shdtz_k } 
            + \vmod{ \shdtz_k - \npk{k} } \leq L\constdmetz +
            \vmod{ \shdtz_k - \npk{k} }.   $$
            Let $k>-N.$ Consider the following cases:

        \begin{itemize}
           \item if $\repz_k > k+\tau$ then
               $$ \dmz(\repz_k,\shadptz) = \flow(\repz_k - (k+\tau), \dmz(k,\shadptz)) =
               \flow(\repz_k - (k+\tau), \npk{k}). $$
               Hence
               $$ \vmod{x_k - \npk{k} } = 
               \vmod{\flow(\repz_k - (k+\tau), \npk{z}) - \npk{k} } 
               \leq \maxshearflow (\repz-k-\tau) \leq 2\maxshearflow L\constdmetz  .$$

           \item If $\repz_k < k-\tau$ then
               \begin{gather*}
                \vmod{x_k - \npk{k} } = 
               \left|  \flow(\repz_k - (k-1+\tau),\dmz(k-1+\tau,\shadptz))  - \right. \\
                \left. - \flow(1 -\tau,\dmz(k-1+\tau,\basept)) -dz_k  \right|= \\
                   = \vmod{\flow(\repz-(k-1+\tau),\npk{k-1}) - 
                  \flow(1 -\tau,\npk{k-1}) } \leq \maxshearflow (\repz_k - k) 
                  \leq \maxshearflow L\constdmetz. 
              \end{gather*}

            \item If $\repz_k \in [k-{\tau},k+{\tau}] $ then
                denote $u = k-1+\tau.$ Thus
               \begin{gather*}
                \vmod{x_k - \npk{k} } = 
                \left| \Gamma( \flow(\repz_k-u,\dmz(u,\shadptz)),\npk{k} ) 
                - \npk{k} \right| \leq \\ 
                 \leq  \vmod{ \flow(\repz_k-u,\dmz(u,\shadptz)) - \npk{k} }.  
               \end{gather*}
                Now we can estimate the above value similarly to the previous case.
        \end{itemize}

        Therefore we just proved that
        \begin{equation}
            \vmod{\npk{k} - \baseptk{k}} \leq \constpknpk d ,\quad k\in\Nints,
            \label{neq:npkbaseptk}
        \end{equation}
        where $\constpknpk$ is a constant that is independent of  $N.$


        Now we prove the solvability of equations using $\dmo$.
        To emphasize the dependence on $d$ we will further
        write upper index $(d).$
        Lipschitz inverse shadowing guarantees that there exists
    a point $\shadpt$ and a reparametrization $\reparam\in\Repe{L\constdmeto}$ such that 
    $$ \vmod{\flow(t,\basept)- \dmsod(\reparam(t), \shadpt)} 
    \leq L\constdmeto \leq \half{r},\quad
        t\in\Rb. $$
        
        
        Choose $d$ small enough for  inequality  \eqref{neq:shadtowork} to work. Then
        \begin{equation}\label{neq:time}
            \vmod{\reparam(k) - k } \leq  L\constdmeto k \leq \frac{\tau}{2},\quad 
            k\in\Ninta.
        \end{equation}

        Denote $\vd = -\basept + \shadpt,$ 
        $$\tdiff_k = \reparam(k+N) - k,\quad \shdt_k = 
        \dmsod(\reparam(k+N),\tvd),\quad k\in\Nints.$$
        Inequality \eqref{neq:time} and the definition of the function  $\gamma$ 
        implies the following equalities: 
        $$ \shdt_k = \omod{k}( \tdiff_k, \vd),\quad k\in\Nints.$$
                Set $\vpreres_k = \dmod(k,\tvd) - \nnpkd{k}$ for $k\in\Nints.$
                Using formula  \eqref{eq:omegadef} we conclude that
                the sequence $\vpreres_k$ solves the following equations:
        $$ \vpreres_{k+1} =   
        \nAkd{k} \vpreres_k + X(\nnpkd{k+1})\tdiff_{k+1} 
            + dz_{k+1},\ k\in\Nint .$$
        Divide them by  $d$ and take
        $$\vres_k = (\shdt_k - \nnpkd{k}) / d,\quad 
            \tdiffoned_k = \tdiff_k/d,\quad k\in\Nints.$$
        Thus
        \begin{equation}
            \vres_{k+1} = \nAkd{k} \vres_k + 
            X(\nnpkd{k+1})\tdiffoned_{k+1} + z_{k+1},\quad k\in\Nint.
             \label{eq:vressolve}
        \end{equation}

        Now we estimate $\vmod{\vpreres_k}:$
        \begin{equation}
             \vmod{\vpreres_k} \leq \vmod{ \dmod(k,\tvd) - \shdt_k} + 
             \vmod{ \shdt_k - \nnpkd{k}},\quad k\in\Nints. 
             \label{neq:vpreres}
        \end{equation}
        At first we deal with the first summand: 
        \begin{gather*}
         \vmod{ \dmod(k,\tvd) - \shdt_k} = 
        \vmod{ \omod{k}(0,\vd) - \omod{k}(\tdiff_k,\vd) } = \\
         = \vmod{ X(\nnpkd{k-1}) \tdiff_k }  \leq 
        \maxvfield L\constdmeto,\quad k\in\Nints.
    \end{gather*}
        
        We can write estimates for the second summand of the right-hand side of  
        \eqref{neq:vpreres}, using inequality
        \eqref{neq:npkbaseptk}:
        \begin{gather*}
        \vmod{ \shdt_k - \nnpkd{k}} \leq \vmod{\baseptk{k} - \nnpkd{k} }
        + \vmod{\shdt_k - \baseptk{k}}  
         \leq \\
        \leq L\constdmeto+ \vmod{\flow(1,\baseptk{k-1}) - 
        \flow(1-\tau,\dmo(k+\tau,\basept) ) } \leq \\
        \leq \vmod{\flow(1,\baseptk{k-1}) - \flow(1-\tau,\npkd{k}) } +
      \vmod{ \flow(1-\tau,\npkd{k}) - \flow(1-\tau,\dmo(k-1+\tau,\basept) ) } \leq \\
      \leq \vmod{X(\nnpkd{k-1}) \tau + A_k(\baseptk{k-1} - \npkd{k-1} )
      + \osmX\enbrace{\vmod{\tau} +\vmod{\baseptk{k-1} - \npkd{k-1}} } } +  \\ 
      + \maxdiv\vmod{\npkd{k-1}-\dmo(k-1+\tau,\basept)} \leq 
      \maxvfield \tau + 2 \maxderiv \constpknpk d + \maxdiv \maxvfield \tau \leq \\
      \leq \maxvfield L\constdmeto + 2 \maxderiv \constpknpk d + 
      \maxdiv \maxvfield L\constdmeto,\quad k\in\Nintw.  
  \end{gather*}

      Denote
        $$L_1 = \oned\enbrace{\maxvfield L\constdmeto +
        \maxvfield L\constdmeto + 2 \maxderiv \constpknpk d + 
        \maxdiv \maxvfield L\constdmeto}. $$

        Therefore
        $$ \vmod{\vres_k} = \vmod{\frac{\vpreres_k}{d}} \leq L_1,\quad k\in\Nints.$$
        Then up to taking a subsequence, there exists the limit 
        $$
        \vfin_k = \lim_{d\to 0} \vres_k, \quad \vmod{\vfin_k} \leq L_1,\quad k\in\Nints.
        $$

        Since by assumption the vector field does not have rest points,
        the values $ \vmod{X(x)} $ are bounded both from above and from  $0$ 
        for any $x\in\mfd.$  
        The values $ \vmod{X(\nnpkd{k+1})\tdiffoned_{k+1} } $ are also bounded  
        due to the fact that all other terms in equality \eqref{eq:vressolve} are bounded.
        Hence the values $\tdiffoned_{k+1}$ 
        are also bounded and converge to  $\tdiffonedfin_{k+1}.$
        Summarizing all these we just have proved that the sequence
        $\seq{\vfin_k}$ satisfies the desired equations
        \begin{gather*}
        \vfin_{k+1} = \ldz \vres_{k+1} = \ldz \nAkd{k} \vres_k + 
        \ldz X(\nnpkd{k+1})\tdiffoned_{k+1} + z_{k+1}=  \\
         = A_k \vfin_k+ X(p_{k+1}) \tdiffonedfin_{k+1} + z_{k+1},\quad k\in\Nint
        \end{gather*}
        and is bounded by a constant $L_1$ that is independent of $N$.
        


\section{Proof of Statement \ref{st:solveqs} for the case of a closed manifold}


Fix a natural $N.$ 
We again will construct a method of the class $\dmetclass{s}$ 
that will \enkav{contain} the equations.
We use the same notation as in the beginning of section  \ref{sec:proofR}.

Let $r_{1}$ be a positive number such that for any 
$x$ the mappings $\exp_x,\ \inv\exp_x$ are defined
and are diffeomprphisms on $B_{r_1}(x)\subset \mfdM$ and
$B_{r_1}(0)\subset T_x \mfdM$ correspondingly.
Withot loss of generality we may assume that
the radius  $r_1$ is small enough so that
the exponential mappings and their inverses distort
distances less then twice:
\begin{IEEEeqnarray*}{rcll}
    \vmod{ \inv\exp_x (y_1) - \inv\exp_x(y_2) } \ & \leq &\  2\dist(y_1, y_2),
    & \quad y_1, y_2 \in B_{r_{1}}(x) \subset M;  \\
    \dist(\exp_x(\tilde{y}_1), \exp_x(\tilde{y}_2) ) \ & \leq &\
    2 \vmod{\tilde{y}_1 - \tilde{y}_2},
    & \quad \tilde{y}_1, \tilde{y}_2 \in B_{r_{1}}(0) \subset T_x M.
\end{IEEEeqnarray*}


We define the method like we did in section \ref{subsec:buildmetR}
with the exception that now we should take into account that we work
on a manifold.

Let $\varkappa$ be equal to either $0$ or $1$ and
$k$ be an integer from the interval $\Nints.$  
Assume that the functions $\dm(k,\cdot),\ \dm(k+\tau,\cdot)$ are already defined. 
Denote
\begin{gather*}
\npk{k} = \dm(k,\basept),\quad \nAk{k} = \diff{\flow(1-\tau,\dm(k+\tau,\basept) )}{x}, \\
 \quad \nnpk{k+1} = \flow(1-\tau,\dm(k+\tau) ),\ \nnpk{\baseptind} = \basept.
\end{gather*}
If $k<N$ we define the function $\om{k+1}:
[-\tau,\tau]\times\tgtM{\baseptind} \to \tgtM{k+1}$ in the following way:
\begin{gather*}
 \om{k+1}(s,v) =  \varkappa \nAk{k} 
\enbrace{ \expi{\npk{k}} \enbrace{\dm(k,\tvM) } + dz_k } +  \\ 
  + \varkappa X ( \nnpk{k+1} ) s + d z_{k+1} - \varkappa d\nAk{k} z_k.
\end{gather*}
        %
        %
        %
We define also the interpolation function
$\Gamma:\mfd\times\mfd\times \mfd \times [-\tau,\tau]\to \mfd :$
\begin{equation*}
    \Gamma(x,y,v,s) =  \gamma(s,\vmod{v})x + (1-\gamma(s,\vmod{v}) ) y  .  
\end{equation*}

Fix a positive  $r$ that is less than $r_1.$ 
Define the method  $\dm$ in the following way:

 \begin{gather*}
\dm(t,x) =  \flow(t,x),    \quad t\in \Rb,\quad x\notin B_r(\basept)\subset \mfdM; \\
    \dm(t,x) =  \flow(t,x),  \quad t\in(-\infty,\tzero-\tau],\quad x\in \mfd. \\
    \dm(k+s,\tvM) =   \\
     =\exp_{\nnpk{k}} \enbrace{ \Gamma \enbrace{ \om{k}(s,v) 
    ,\expi{\nnpk{k} } \enbrace{ \flow(1-\tau+s, \dm(k-1+\tau,\tvM) ) },v,s } } , \\
    k\in[-N+1,N],\quad s\in[-\tau,\tau],\quad v\in B_r(0)\subset T_{\basept} \mfdM;\\
    \dm(k+\tau+s,\tvM) =   \flow(s,\dm(k+\tau,\tvM)), \\
    k\in[-N+1,N-1],\quad s\in[0,1-2\tau],\quad v\in B_r(0)\subset T_{\basept} \mfdM; \\
    \dm(N+\tau+s,\tvM) =   \flow(s,\dm(k+\tau,\tvM)), \\
    s\in[0,\infty),\quad v\in B_r(0)\subset T_{\basept} \mfdM. 
\end{gather*}

Unlike the case of  $\mfdM=\mfd$ here we need some additional assumptions 
for the definition to be correct. 
We need all the mappings $\exp$ and $\inv{\exp}$ that were used to be defined. 
 It is easy to see that it is enough to guarantee that the following inequalities
 are satisfied:
\begin{IEEEeqnarray*}{c}            
    \dist\enbrace{\npk{k}, \dm(k,\tvM) } \leq r_1,\quad k\in\Nints; \\
    \dist\enbrace{\nnpk{k},  \flow(1-\tau+s, \dm(k-1+\tau,\tvM) ) } \leq r_1, \quad k\in\Nint; \\
 \vmod{\Gamma \enbrace{ \om{k}(s,v)
    ,\expi{\nnpk{k} } \enbrace{ \flow(1-\tau+s, \dm(k-1+\tau,\tvM) ) },v,s } } \leq r_1,\\ 
    k\in\Nint.
\end{IEEEeqnarray*}

One can verify that the right-hand sides of these inequalities can be
bounded from above by $10 \maxderiv ^ {2N} (d + r + \tau \maxvfield).$ 
This means that it is enough to take 
 $d,r$ and $\tau$ small enough.

The rest of the proof is fully analogous to the case of  $M=\Rb^d.$

\section*{Acknowledgements}

 Research was supported by RFBR (project 12-01-00275) and the Chebyshev laboratory (grant of the Russian government N 11.G34.31.0026).

 I am grateful to S. Yu. Pilyugin and S. Tikhomirov for their helpful comments and advice
 and to A. Petrov for discussions.

\bibliography{bibrefs}
\bibliographystyle{cDSS.bst}


\end{document}